\documentclass[11pt]{amsart}

\usepackage{amssymb,amsfonts,latexsym}
\usepackage[centertags]{amsmath}
\usepackage{amsfonts}
\usepackage{amssymb}
\usepackage{amsthm}
\usepackage[all]{xy}
\usepackage{cite}
\usepackage{graphicx,color}
\usepackage{todo}

\newtheorem*{Prob*}{Problem}
\newtheorem{cont}{cont}[section]
\newtheorem{theorem}[cont]{Theorem}
\newtheorem*{teo*}{Theorem}
\newtheorem{proposition}[cont]{Proposition}
\newtheorem{lemma}[cont]{Lemma}
\newtheorem{corollary}[cont]{Corollary}
\newtheorem{defn}[cont]{Definition}
\newtheorem{pblm}[cont]{Problem}

\newtheorem*{Enunciato*}{Enunciato}

\numberwithin{equation}{section}
\newtheorem{finalremark}[cont]{Final Remark}

\newtheorem*{not*}{Notation}

\newcommand{\cO}{{\mathcal O}}
\newcommand{\shF}{\mathcal{F}}

\newcommand{\shH}{\mathcal{H}}

\newcommand{\cL}{\mathcal{L}}
\newcommand{\shM}{\mathcal{M}}
\newcommand{\shE}{\mathcal{E}}
\newcommand{\PP}{\mathbb{P}}
\newcommand{\ZZ}{\mathbb{Z}}

\newcommand{\odi}[1]{\mathcal{O}_{#1}}

\newcommand{\arr}{\longrightarrow}

\DeclareMathOperator{\Hl}{H} \DeclareMathOperator{\h}{h}
 
 \DeclareMathOperator{\Hom}{Hom}
 \DeclareMathOperator{\Aut}{Aut}

\DeclareMathOperator{\Mat}{Mat}
\DeclareMathOperator{\depth}{depth}
\DeclareMathOperator{\Proj}{Proj} \DeclareMathOperator{\di}{dim}
\DeclareMathOperator{\codim}{codim}
\DeclareMathOperator{\Ext}{Ext} \DeclareMathOperator{\pd}{pd}
 
 \DeclareMathOperator{\im}{im}

\begin{document}

\subjclass[2010]{Primary 14F99; Secondary 14J99}
\keywords{Representation type, Arithmetically Cohen-Macaulay bundles, Geometrically wild}

\title{On the representation type of a projective variety}

\author[Rosa M. Mir\'{o}-Roig]{Rosa M. Mir\'{o}-Roig$^{*}$}
\address{Facultat de Matem\`atiques, Department d'Algebra i Geometria, Gran Via des les Corts Catalanes 585, 08007 Barcelona, Spain}
\email{miro@ub.edu}

\date{\today}
\thanks{$^*$ Partially supported by MTM2010-15256.}

\begin{abstract}  Let $X\subset \PP^n$ be a smooth arithmetically Cohen-Macaulay variety.
We prove that the restriction $\nu _{3|X}$ to $X$ of the Veronese 3-uple embedding
$\nu _3:\PP^n \longrightarrow \PP^{{n+3\choose 3}-1}$ embeds $X$ as a variety
of wild representation type.
\end{abstract}

\maketitle

\tableofcontents

\section{Introduction}

\vskip 2mm The importance of the existence of arithmetically Cohen-Macaulay  bundles (i.e. bundles without intermediate cohomology)  on a non-singular projective variety
relies on the fact that a natural way to measure the complexity of a non-singular projective variety is to ask for the family of non-isomorphic indecomposable arithmetically Cohen-Macaulay (shortly, ACM) bundles that it supports. This problem has a long and interesting history behind.  A seminal result is due to Horrocks (cf. \cite{Hor}) who asserted
that, up to twist, there is only one indecomposable ACM bundle on $\PP^n$: $\odi{\PP^n}$. This corresponds with the general idea that a "simple"
variety should have associated a "simple" category of ACM bundles. Following these lines, a cornerstone result was the classification  of ACM varieties of {\em finite representation type}, i.e., varieties that support (up to twist and isomorphism) only a finite number of indecomposable ACM bundles. It turned out that they fall into a very short list: $\PP^n$, a smooth hyperquadric $Q\subset \PP^n$,  a cubic scroll in $\PP^4$, the Veronese surface in $\PP^5$, a rational normal curve and three or less reduced points in  $\PP^2$  (cf. \cite[Theorem C]{BGS} and \cite[p. 348]{EH}).

 \vskip 2mm For the rest of ACM varieties, it became an interesting problem to  give a criterium to split them into a finer classification.  Inspired in Representation Theory, it has been proposed the classification of ACM varieties  as {\em finite, tame or wild} (see Definition \ref{wild})  according to the complexity of their associated category of ACM bundles.
So far only few examples of varieties of wild
representation type are known: curves of genus $g\geq 2$ (cf. \cite{DG}), del Pezzo
surfaces and Fano blow-ups of points in $\PP^n$ (cf. \cite{MR4}, the cases of the
cubic surface and the cubic threefold have also been handled in \cite{CH1}), ACM
rational surfaces on $\PP^4$ (cf. \cite{MR1}), any Segre variety unless the quadric
 surface in $\PP^3$ (cf. \cite[Theorem 4.6]{CMP}), and any rational normal scroll unless $\PP^n$,
 the rational normal curve, the  quadric surface in $\PP^3$ and the cubic scroll in $\PP^4$
 which are of finite representation type (cf. \cite[Theorem 3.8]{MR}).

\vskip 4mm

The representation type of an ACM variety strongly depends on the chosen polarization. For instance, we have mentioned that the quadric surface $X\cong \PP^1\times \PP^1\subset \PP^3$ is a variety of finite representation type with respect to the very ample line bundle $\cO _{\PP^1\times \PP^1}(1,1)$. However, the smooth quadric surface $X$ embedded in $\PP^8$ through the very ample anticanonical  divisor $\cO _{\PP^1\times \PP^1}(2,2)$
is of wild representation type (cf. \cite[Theorem 4.10]{MR4}). This leads to  the following problem

\begin{pblm} \label{problem} (a)  Given an ACM variety $X\subset \PP^n$ , is there an integer $N_X$ such that $X$ can be embedded in $\PP^{N_X}$ as a variety of wild representation type?

(b) If so, what is the smallest possible integer $N_X$?
\end{pblm}

The goal of this short note is to answer affirmatively Problem (a) and to provide an upper bound for $N_X$.   In other words, for any
smooth projective variety $X$ there is an embedding of $X$ into a projective space $\PP^{N_X}$ such that the corresponding homogeneous coordinate ring has
arbitrary large families of non-isomorphic indecomposable graded Maximal Cohen-
Macaulay modules. Actually, it is proved that such an embedding can be
obtained as the composition of the "original" embedding $X \subset \PP^n$ and the
Veronese 3-uple embedding $\nu _3:\PP^n \longrightarrow \PP^{{n+3\choose 3}-1}$.
The idea will be to construct on any ACM  variety $X\subset \PP^n$ of dimension $d\ge 2$ irreducible families $\shF$
of vector bundles $\shE$ of arbitrarily rank and dimension with
the extra feature that  any $\shE \in \shF$
satisfy $\Hl^{i}(X,\shE (t))=0$ for all $t \in \ZZ$ and $2\le i \le d-1$ and $\Hl^1(X,\shE(t))=0$ for all $t\ne -1,-2$. Therefore, $X$ embedded in $\PP^{\h^0(\odi{X}(s))-1}$ through the very ample line bundle $\odi{X}(s)$, $s\ge 3$, is of wild representation type.

\vskip 2mm
Let us outline the structure of this paper. In section 2, we recall the definitions and basic facts on ACM varieties and  bundles need later. Section 3 is the heart of the paper and contains our main result (cf. Theorem \ref{main})

\vskip 2mm
\noindent {Notation.} Throughout this paper $K$ will be an algebraically closed field of characteristic zero, $R=K[x_0, x_1, \cdots ,x_n]$, $\mathfrak{m}=(x_0, \ldots,x_n)$ and $\PP^n=\Proj(R)$.
Given a non-singular variety $X$ equipped with an ample line bundle $\odi{X}(1)$, the line bundle $\odi{X}(1)^{\otimes l}$ will be denoted by $\odi{X}(l)$. For any coherent sheaf $\shE$ on $X$ we are going to denote the twisted sheaf $\shE\otimes\odi{X}(l)$ by $\shE(l)$. As usual, $\Hl^i(X,\shE)$ stands for the cohomology groups, $\h^i(X,\shE)$ for their dimension and $\Hl^i_*(X,\shE)=\oplus _{l \in \ZZ}\Hl^i(X,\shE(l))$.

\vskip 2mm \noindent {\bf Acknowledgement.} The author wishes to thank Laura Costa and Joan Pons-Llopis for  many useful
 discussions on this subject. 

\section{Preliminaries}

 We set up here some preliminary notions mainly concerning  the definitions and basic results on ACM  schemes $X\subset \PP^n$  as well as on ACM sheaves $\shE$ on $X$ needed  in the sequel.

\begin{defn}\rm  A subscheme  $X\subseteq \PP^n$  is said to be arithmetically
Cohen-Macaulay (briefly, ACM)  if  its homogeneous
coordinate ring $R_X=R/I_X$ is a Cohen-Macaulay ring, i.e. $\depth
(R_X)=\dim (R_X)$.
\end{defn}

Thanks to the graded version of the
Auslander-Buchsbaum formula (for any finitely generated
$R$-module $M$): $$\pd (M)=n+1-\depth (M),$$ we deduce that a
subscheme $X\subseteq \PP^n$ is ACM if and only if
$ \pd (R_X)=\codim X$.
Hence, if $X\subseteq \PP^n$ is a codimension $c$ ACM subscheme, a
graded minimal free $R$-resolution of $I_X$ is of the form:
$$ 0\longrightarrow F_c   \stackrel{\varphi _c}{\longrightarrow}  F_{c-1}\stackrel{\varphi _{c-1}}{\longrightarrow}  \cdots
\stackrel{\varphi _2}{\longrightarrow}  F_1 \stackrel{\varphi _1}{\longrightarrow} F_0 \longrightarrow R_X\longrightarrow 0$$ with $F_0=R$ and
$F_{i}=\oplus _{j=1}^{\beta _i}R(-n_{j}^{i})$, $1\le i  \le c$
(in this setting, minimal means that $\im \varphi _{i}\subset
\mathfrak{m}F_{i-1}$).

\vskip 2mm

\begin{defn}\label{ACM} \rm
Let $X\subset \PP^n$ be a smooth  projective variety with a very ample invertible  sheaf $\cL$. A coherent sheaf $\shE$ on $X$ is \emph{Arithmetically Cohen Macaulay} (ACM for short) with respect to $\cL$ if
$$
\Hl^i(X,\shE \otimes \cL ^{\otimes t})=0 \quad\quad \text{    for all $1\le i\le \di X-1$ and $t\in \ZZ$.}
$$
Often, when $\cL =\odi{X}(1)$ we will omit it and we will simply say that $\shE$ is ACM.
\end{defn}

A possible way to classify ACM varieties is according to the complexity of the category of ACM sheaves that they support.
Recently, inspired by an analogous classification for quivers and for $K$-algebras of finite type, it has been proposed the classification of any  ACM variety as being of \emph{finite, tame or wild representation type} (cf. \cite{DG}). Let us introduce these definitions slightly modified with respect to the usual one:

\begin{defn} \label{wild} \rm Let $X\subseteq\PP^N$ be an ACM scheme of dimension $n$.

 (i) We say that $X$ is of \emph{finite representation type} if it has, up to twist and isomorphism, only a finite number of indecomposable ACM sheaves.

(ii)  $X$ is of \emph{tame representation type} if either it has, up to twist and isomorphism, an infinite
 discrete set of indecomposable ACM sheaves or,
 for each rank $r$, the indecomposable ACM sheaves of rank $r$ form a finite number of families of dimension at most $n$.

 (iii) $X$ is of \emph{wild representation type} if there exist $l$-dimensional families of non-isomorphic indecomposable ACM sheaves for arbitrary large $l$.
\end{defn}

The problem of classifying ACM varieties  according to the complexity of the category
of ACM sheaves that they support has recently  attired much attention and, in particular, the following problem is
 still open:

\begin{pblm}  Is the trichotomy finite representation type, tame representation type  and wild representation type exhaustive?
\end{pblm}

 One of the main achievements in this
field has been the classification of  varieties of finite representation type (cf.. \cite[Theorem C]{BGS} and \cite[p. 348]{EH}));
it turns out
that they fall into a very short list: three or less reduced points on $\PP^2$, a projective space, a non-singular quadric hypersurface $X\subseteq \PP^n$, a cubic scroll in $\PP^4$, the Veronese surface in $\PP^5$ or a rational normal curve.
As examples of a variety of tame representation type we have the elliptic curves and
the quadric cone in $\PP^3$ (cf.. \cite[Proposition 6.1]{CaH}).  Finally, on the other extreme of complexity lie those varieties that have very large
families of ACM sheaves. So far only few examples of varieties of wild representation type are known: curves of genus $g\geq 2$ (cf. \cite{DG}), del Pezzo surfaces and Fano blow-ups of points in $\PP^n$ (cf. \cite{MR4}, the cases of the cubic surface and the cubic threefold have also been handled in \cite{CH1}), ACM rational surfaces on $\PP^4$ (cf. \cite{MR1}), Segre varieties other than the quadric in $\PP^3$ (cf. \cite[Theorem 4.6]{CMP}), rational normal scrolls other than $\PP^n$, the rational normal curve  in $\PP^n$, the quadric
 in $\PP^3$ and the cubic scroll  in $\PP^4$ (cf. \cite[Theorem 3.8]{MR}) and hypersurfaces $X\subset \PP^n$ of degree $\ge 4$ (cf. \cite[Corollary 1]{To}).

As we pointed out in the introduction, the representation type of an ACM variety $X\subset \PP^n$ strongly depends on the chosen polarization and our goal will be to prove that on a projective variety $X\subset \PP^n$ there always exists a very ample line bundle $\cL$ on $X$ which embeds $X$ in $\PP^{\h^0(X,\cL)-1}$ as a variety  of wild representation type (cf Theorem \ref{main}). As immediate consequence we will have many new examples
of ACM varieties of wild representation type.


\section{The representation type of an ACM variety}

In this section, $X$ will be a smooth ACM variety  of dimension $d\ge 2$ in $\PP^n$ with a minimal free $R$-resolution of the following type:
\begin{equation}\label{seq} 0\longrightarrow F_c   \stackrel{\varphi _c}{\longrightarrow}
F_{c-1}\stackrel{\varphi _{c-1}}{\longrightarrow}  \cdots
\stackrel{\varphi _2}{\longrightarrow}  F_1 \stackrel{\varphi _1}{\longrightarrow} F_0
\longrightarrow R_X\longrightarrow 0\end{equation} with $c=n-d$, $F_0=R$ and
$F_{i}=\oplus _{j=1}^{\beta _i}R(-n_{j}^{i})$, $1\le i  \le c$.
Our goal is to prove that the Veronese embedding $\nu _3 :\PP^n \longrightarrow \PP^{{n+3\choose 3}-1}$  embeds $X$  as
a variety  of wild representation type. The idea will be to construct on $X$ families
of undecomposable  vector bundles of arbitrary rank and dimension which will be ACM
with respect to a $\odi{X}(3)$ but not necessarily
with respect to $\odi{X}(1)$.
The ACM bundles $\shE$ on $X$ will be constructed as kernels of certain surjective maps
between $\cO _X(t)^a$ and $\cO _X(t+1)^b$ for suitable values of $t, a ,b \in \ZZ$.

\vskip 2mm
To this end,
let us fix some notation as presented in \cite{EHi}. We consider $K$-vector spaces $A$
and $B$ of dimension $a$ and $b$, respectively. Set $V=\Hl^0(\PP^n,\odi{\PP^n}(1))$
and let $M=\Hom(B,A\otimes V)$ be the space of ($a\times b$)-matrices of linear forms.
It is well-known that there exists a bijection between the elements $\phi\in M$ and
the morphisms $$\phi:B\otimes\cO_ {\PP^n}\arr A\otimes\cO _ {\PP^n}(1).$$ Taking the
tensor with $\cO _{\PP^n}(1)$ and considering global sections, we have morphisms
$$\Hl^0(\phi(1)):\Hl^0(\PP^n, \cO _{\PP^n}(1)^b)\arr \Hl^0(\PP^n, \cO _{\PP^n}(2)^a).$$
The following result tells us under which conditions the aforementioned morphisms
$\phi$ and $\Hl^0(\phi(1))$ are surjective:

\begin{proposition}\label{hirschowitz}
For $a\geq 1$, $b\geq a+n$ and $2b\geq (n+2)a$, the set of elements $\phi\in M$ such
that $$\phi:B\otimes\odi{\PP^n}\arr A\otimes\cO _{\PP^n}(1) \text{ and }
\Hl^0(\phi(1)):\Hl^0(\PP^n, \odi{\PP^n}(1)^b)\arr \Hl^0(\PP^n, \odi{\PP^n}(2)^a)$$ are
surjective forms a non-empty open dense subset that we will denote by $V_{n}$.
\end{proposition}
\begin{proof} See \cite[Proposition 4.1]{EHi}.
\end{proof}

For any $2\le n$  and any $1\le a$, we denote by  $\shE _{n,a}$ any vector bundle on
$\PP^n$ given by  the exact sequence
\begin{equation}\label{eq1}
0\arr\shE
_{n,a}\arr\odi{\PP^n}(1)^{(n+2)a}\stackrel{\phi(1)}{\arr}\odi{\PP^n}(2)^{2a}\arr 0
\end{equation}
\noindent where $\phi \in V_n$.  Note that $\shE _{n,a}$   has rank $na$.

\begin{lemma}\label{cohomgroups} With the above notation we have:
\begin{itemize}
\item[(i)] $$\h^0(\PP^n,\shE _{n,a}(t))=\begin{cases} 0 &  \text{for } t\leq 0,
    \\
a((n+2)\binom{n+t+1}{n}-2\binom{n+t+2}{n}) &  \text{for } t>0.
\end{cases} $$
\item[(ii)] $$\h^1(\PP^n,\shE _{n,a}(t))=\begin{cases} 0 &  \text{ for } t<-2
    \text{ or } t\geq 0, \\ an &  \text{ for } t=-1 \\ 2a &  \text{ for } t=-2.
    \end{cases}$$
\item[(iii)] $\h^i(\PP^n,\shE _{n,a}(t))=0$ for all $t\in\ZZ$ and $2\leq i\leq
    n-1$.
\item[(iv)] $\h^n(\PP^n,\shE _{n,a}(t))=0$ for $t\geq -n-1$.
\end{itemize}
\end{lemma}

\begin{proof}
Since $\phi \in V_n$, by Proposition \ref{hirschowitz},
$\Hl^0(\phi(1))$ is surjective. But, since the $K$-vector spaces $\Hl^0(\PP^n,
\odi{\PP^n}(1)^{(n+2)a})$ and
$\Hl^0(\PP^n,\odi{\PP^n}(2)^{2a})$
have the same dimension, $\Hl^0(\phi(1))$ is an isomorphism and therefore $\Hl^0(\shE
_{n,a})=0$. \emph{A fortiori}, $\Hl^0(\shE_{n,a}(t))=0$ for $t\leq 0$. On the other
hand, again by the surjectivity of $\Hl^0(\phi(1))$, $\Hl^1(\shE_{n,a})=0$. Since it
is obvious that $\Hl^i(\shE_{n,a}(1-i))=0$ for $i\geq 2$ it turns out that $\shE
_{n,a}$ is 1-regular and in particular, $\Hl^1(\shE_{n,a}(t))=0$ for $t\geq 0$. The
rest of cohomology groups can be easily deduced from the long exact cohomology
sequence associated to the exact sequence (\ref{eq1}).
\end{proof}

From now on, for any $2\le n$  and any $1\le a$, we call $\shF ^X_{n,a}$ the family of {\em general} rank $na$ vector bundles $\shE $ on $X\subset \PP^n$ sitting in an exact sequence of the following
type:

\begin{equation}\label{defseq}
0\arr\shE \arr\odi{X}(1)^{(n+2)a}\stackrel{f}{\arr}\odi{X}(2)^{2a}\arr 0.
\end{equation}

\begin{proposition} \label{key} Let $X\subset \PP^n$ be a smooth ACM variety of dimension $d\ge 2$. With the above notation, we have:
\begin{itemize}
\item[(1)] A general vector bundle $\shE \in \shF ^X_{n,a}$ satisfies $$
        \begin{array}{lcc} \Hl^{i}_{*}\shE=0 & \text{ for } 2\le i \le d-1, \\
    \Hl^1(X,\shE(t))=0 & \text{ for } t\ne -1,-2 .\end{array} $$
\item[(2)]  A general vector bundle $\shE \in \shF ^X_{n,a}$ is simple.
\item [(3)] $\shF ^X_{n,a}$ is a non-empty irreducible family of dimension
    $a^2(n^2+2n-4)+1$ of  simple (hence indecomposable) rank $an$ vector bundles
    on $X$.
    \end{itemize}
\end{proposition}
\begin{proof} (1) Since  $\Hl^{i}(X,\shE(t))=0$ for all $t\in \ZZ$ and $ 2\le i \le
d-1$, and $\Hl^1(X,\shE(t))=0$  for $t\ne -1,-2$ are open conditions, it is enough to
exhibit a vector bundle $\shE \in \shF ^X_{n,a}$ verifying these vanishing.
Tensoring the exact sequence (\ref{eq1}) with $\odi{X}$, we get
\begin{equation}
0\arr \shE:= \shE _{n,a}\otimes \odi{X}\arr\odi{X}(1)^{(n+2)a}\arr
\odi{X}(2)^{2a}\arr 0.
\end{equation}
Taking cohomology, we immediately obtain $\Hl^{i}(X,\shE(t))=0$ for all $t\in \ZZ$ and
$ 2\le i \le d-1$.  On the other hand, we tensor with $\shE _{n,a}$ the exact sequence (\ref{seq}) sheafiffied

$$0\longrightarrow \oplus _{j=1}^{\beta _c}\odi{\PP^n}(-n_{j}^{c})   \stackrel{\varphi _c}{\longrightarrow}
\oplus _{j=1}^{\beta _{c-1}}\odi{\PP^n}(-n_{j}^{c-1})
\stackrel{\varphi _{c-1}}{\longrightarrow}  \cdots
\stackrel{\varphi _2}{\longrightarrow}  \oplus _{j=1}^{\beta _1}\odi{\PP^n}(-n_{j}^{1})   \stackrel{\varphi _1}{\longrightarrow} \odi{\PP^n}\stackrel{\varphi _0}{\longrightarrow}
 \odi{X}\longrightarrow 0 $$ and we get
\begin{equation}\label{seq2} 0\longrightarrow \oplus _{j=1}^{\beta _c}\shE _{n,a}(-n_{j}^{c})   \stackrel{\varphi _c}{\longrightarrow}
\oplus _{j=1}^{\beta _{c-1}}\shE _{n,a}(-n_{j}^{c-1})
\stackrel{\varphi _{c-1}}{\longrightarrow}  \cdots
\stackrel{\varphi _{i+1}}{\longrightarrow}
\oplus _{j=1}^{\beta _{i}}\shE _{n,a}(-n_{j}^{i})\stackrel{\varphi _{i}}{\longrightarrow}
\end{equation}
$$
\cdots \stackrel{\varphi _2}{\longrightarrow}  \oplus _{j=1}^{\beta _1}\shE _{n,a}(-n_{j}^{1})   \stackrel{\varphi _1}{\longrightarrow} \shE _{n,a}\stackrel{\varphi _0}{\longrightarrow}
 \shE= \shE _{n,a}\otimes \odi{X}  \longrightarrow 0. $$ Set $\shH_i:=\ker(\varphi _{i})$, $0\le i \le c-2$. Cutting the exact sequence (\ref{seq2}) into short exact sequences and taking cohomology, we obtain
$$ \cdots \arr \Hl^1(\PP^n,\shE _{n,a}(t)) \arr \Hl^1(X,\shE(t)) \arr \Hl^2(\PP^n,\shH
_{0}(t)) \arr \cdots ,$$
$$ \cdots \arr \Hl^2(\PP^n,\oplus  _{j=1}^{\beta _1}\shE _{n,a}(-n_{j}^{1}+t)) \arr \Hl^2(\PP^n,\shH _0(t)) \arr \Hl^3(\PP^n,\shH
_{1}(t)) \arr \cdots ,$$
$$ \cdots $$
$$ \cdots \arr \Hl^{c-1}(\PP^n,\oplus  _{j=1}^{\beta _{c-2}}\shE _{n,a}(-n_{j}^{c-2}+t)) \arr \Hl^{c-1}(\PP^n,\shH _{c-3}(t)) \arr \Hl^c(\PP^n,\shH
_{c-2}(t)) \arr \cdots ,$$
$$ \cdots \arr \Hl^{c}(\PP^n,\oplus  _{j=1}^{\beta _{c-1}}\shE _{n,a}(-n_{j}^{c-1}+t)) \arr \Hl^{c}(\PP^n,\shH _{c-2}(t)) \arr \Hl^{c+1}(\PP^n, \oplus  _{j=1}^{\beta _{c}}\shE _{n,a}(-n_{j}^{c}+t)) \arr \cdots ,$$

\noindent Using Lemma \ref{cohomgroups}, we conclude that $\Hl^1(X,\shE(t))=0$  for $t\ne -1,
-2$.

(2) A general vector bundle $\shE \in \shF ^X_{n,a}$ sits in an exact sequence $$0\arr
\shE \stackrel{g}{\arr}\odi{X}(1)^{(n+2)a}\stackrel{f}{\arr}
\odi{X}(2)^{2a}\arr 0$$ and to check that $\shE$ is simple is equivalent to check that
$\shE ^{\vee}$ is simple.
Notice that the morphism $f^{\vee} :\odi{X}(-2)^{2a}\longrightarrow
\odi{X}(-1)^{(n+2)a}$ appearing in the exact sequence
\begin{equation} \label{dualseq} 0\arr \odi{X}(-2)^{2a}\stackrel{f^{\vee}}{\arr}\odi{X}(-1)^{(n+2)a}
\stackrel{g^{\vee}}{\arr}
\shE ^{\vee}
\arr 0\end{equation}
is a general element of the $K$-vector space
$$M:=\Hom(\odi{X}(-2)^{2a},\odi{X}(-1)^{(n+2)a})\cong K^{n+1}\otimes K^{2a}\otimes
K^{(n+2)a}$$ because $\Hom(\odi{X}(-2),\odi{X}(-1))\cong\Hl^0(\odi{X}(1))\cong
\Hl^0(\odi{\PP^n}(1))\cong K^{n+1}$. Therefore, $f^{\vee }$ can be represented by a
$(n+2)a\times 2a$ matrix $A$ with entries in $\Hl^0(\odi{\PP^n}(1))$. Since
$\Aut(\odi{X}(-1)^{(n+2)a})\cong GL((n+2)a)$ and $\Aut(\odi{X}(-2)^{2a})\cong GL(2a)$,
the group $GL((n+2)a)\times GL(2a)$ acts naturally on $M$ by

\begin{center}
\begin{tabular}{ c c l}
$GL((n+2)a)\times GL(2a)\times M$ & $\longrightarrow$ & $M$ \\
$(g_1,g_2,A)$  & $\mapsto$ & $g_1^{-1}Ag_2$.\\
\end{tabular}
\end{center}

For all $A\in M$ and $\lambda\in K^*$, $(\lambda Id_{(n+2)a},\lambda Id_{2a})$ belongs
to the stabilizer of $A$ and, hence, $\di_K Stab(A)\geq 1$. Since
$(2a)^2+(n+2)^2a^2-2a(n+1)(n+2)a<0$, it follows from   \cite[Theorem 4]{Kac}  that
$\di_K Stab(A)= 1$.  We will now check that $\shE^{\vee}$ is simple.
Otherwise, there exists a non-trivial morphism $\phi:\shE^{\vee}\arr\shE^{\vee}$ and composing with $g^{\vee}$  we get a morphism $$\overline{\phi}=\phi\circ g^{\vee}:\odi{X}(-1)^{(n+2)a}\arr\shE^{\vee}.$$ Applying  $\Hom(\odi{X}(-1)^{(n+2)a}, -)$ to the exact sequence(\ref{dualseq}) and taking into account
that $$\Hom(\odi{X}(-1)^{(n+2)a},\odi{X}(-2)^{2a})=
\Ext^1(\odi{X}(-1)^{(n+2)a},\odi{X}(-2)^{2a})=0$$ we obtain $\Hom(\odi{X}(-1)^{(n+2)a},\odi{X}(-1)^{(n+2)a})
\cong\Hom(\odi{X}(-1)^{(n+2)a},\shE^{\vee})$. Therefore, there is a non-trivial morphism $\widetilde{\phi}\in\Hom(\odi{X}(-1)^{(n+2)a},\odi{X}(-1)^{(n+2)a})$ induced by $\overline{\phi}$ and represented by a matrix $B \neq \mu Id\in \Mat_{(n+2)a\times (n+2)a}(K)$ such that the following diagram commutes:
$$
\xymatrix{
0 \ar[r] &  \odi{X}(-2)^{2a} \ar[r]_{f^{\vee}} \ar[d]^{C} & \odi{X}(-1)^{(n+2)a}  \ar[d]^{B} \ar[dr]^{\overline{\phi}} \ar[r]_{g^{\vee }}& \shE \ar[r] \ar[d]^{\phi} & 0\\
0 \ar[r] &  \odi{X}(-2)^{2a} \ar[r]_{f^{\vee}} & \odi{X}(-1)^{(n+2)a} \ar[r]_{g^{\vee}} & \shE ^{\vee}\ar[r] & 0\\
}
$$
\noindent where $C\in \Mat_{2a\times 2a}(K)$ is the matrix associated to $\widetilde{\phi}_{|\odi{X}(-2)^{2a}}$. Then the pair $(C,B)\neq (\mu Id,\mu Id)$ verifies $AC=BA$. Let us consider an element $\alpha\in K$ that does not belong to the set of eigenvalues of $B$ and $C$. Then the pair $(B-\alpha Id,C-\alpha Id)\in GL((n+2)a)\times GL(2a)$ belongs to $Stab(f)$ and therefore $\di_K Stab(f)>1$ which is a contradiction. Thus, $\shE$ is simple.

(3) It only remains to compute the dimension of $\shF ^X_{n,a}$. Since the isomorphism class of a general
vector bundle $\shE \in \shF ^X_{n,a}$  associated to a morphism
$\phi\in M:=\Hom(\odi{X}^{(n+2)a},\odi{X}(1)^{2a})$ depends only on the orbit of $\phi$ under the action of $GL((n+2)a)\times GL(2a)$ on $M$, we have:
$$\begin{array}{lcl}
\di \shF ^X_{n,a} & = & \di M - \di\Aut(\odi{X}^{(n+2)a})-\di\Aut(\odi{X}(1)^{2a})+1
\\
& = & 2a^2(n+2)(n+1)-a^2(n+2)^2-4a^2+1=a^2(n^2+2n-4)+1. \end{array}$$
\end{proof}

We are now ready to prove the main result of this short paper and give an affirmative answer to Problem \ref{problem}.
\begin{theorem}\label{main}
Let $X\subset \PP^n$ be a smooth ACM variety of dimension $d\ge 2$. The very ample line bundle $\odi{X}(s)$, $s\ge 3$, embeds $X$ in $\PP^{\h^0(\odi{X}(s))-1}$ as a variety of wild representation type.
\end{theorem}

\begin{proof}
Indeed, given any integer $p$, we choose an integer $a$ such that $an\ge p$. By Proposition \ref{key}, there exists a family $\shF ^X_{n,a}$ of dimension $a^2(n^2+2n-4)+1$ of simple (hence undecomposable) vector bundles $\shE$ on $X$ of rank $an$. For a general $\shE \in \shF ^X_{n,a}$, we have $\Hl^{i}(X,\shE(t))=0$ for all $t\in \ZZ$ and $ 2\le i \le d-1$, and   $\Hl^1(X,\shE(t))=0$ for $t\ne -2,-1$. Therefore,  the very ample line bundle $\cO _{X}(s)$ embeds $X$ in $\PP^{\h^0(\odi{X}(s))-1}$,
as a variety  of wild representation type.
\end{proof}

\begin{corollary} The smallest possible integer $N_X$ such that $X$ embeds as a variety of wild representation type is bounded by $N_X\le {n+3\choose 3}-1$.
\end{corollary}

We will finish the paper with a final remark concerning the terminology.

\begin{finalremark}\rm  In \cite[Definition 1.4]{DG}, Drozd and Greuel  introduced two definitions of wildness: geometrically wild and algebraically wild. Roughly speaking a projective variety $X\subset \PP^n$ is {\em geometrically wild} if the corresponding homogeneous coordinate ring $R_X$ has arbitrarily large families of indecomposable Maximal Cohen-Macaulay  $R_X$-modules. $X$ is said to be {\em algebraically wild} if
for every finitely generated $k$-algebra $A$ there exists a family
of Maximal Cohen-Macaulay $R_X$-modules $\shM$ such that the following conditions
hold:
\begin{itemize}
\item[(1)] For every indecomposable $A$-module $L$ the $R_X$-module $\shM\otimes _A L$ is indecomposable.
\item[(2)] If $\shM\otimes _A L \cong \shM\otimes _A L'$
 for some finite dimensional $A$-modules
$L$ and $L'$, then $L\cong L'$.
\end{itemize}

It is not difficult to check that if $X$ is algebraically wild then it is also
geometrically wild. It is not known though conjectured whether the converse is true, i.e. if geometrically wild implies algebraically wild.

It is worthwhile to point out that the constructed embedding proves the geometrical wildness of $X\subset \PP^{N_X}$ but it remains open whether it is also algebraically  wild.
\end{finalremark}

\providecommand{\bysame}{\leavevmode\hbox to3em{\hrulefill}\thinspace}
\providecommand{\MR}{\relax\ifhmode\unskip\space\fi MR }
\providecommand{\MRhref}[2]{%
  \href{http://www.ams.org/mathscinet-getitem?mr=#1}{#2}
}
\providecommand{\href}[2]{#2}

\end{document}